\documentclass[11 pt]{article}
\usepackage{amsmath}
\usepackage{amsthm}
\usepackage{amssymb}
\usepackage{graphicx}
\usepackage{hyperref}
\usepackage{color}

\newtheorem{theorem}{Theorem}[section]
\newtheorem{definition}[theorem]{Definition}

\newtheorem{lemma}[theorem]{Lemma}

\newtheorem{remark}[theorem]{Remark}
\numberwithin{equation}{section}

\newcommand{\R}{\mathbb{R}}
\renewcommand{\S}{\mathbb{S}}

\DeclareMathOperator{\Tr}{Tr}

\begin{document}

\title{\textsc{Infinitesimal generators for two-dimensional L\'evy process-driven hypothesis testing}  }
\author{ Michael Roberts\footnote{Email: michael.roberts.1@ndus.edu}, Indranil SenGupta\footnote{Email: indranil.sengupta@ndsu.edu} \\ Department of Mathematics \\ North Dakota State University \\ Fargo, North Dakota, USA.}
\date{\today}
\maketitle

\begin{abstract}

In this paper, we present the testing of four hypotheses on two streams of observations that are driven by L\'evy processes. This is applicable for sequential decision making on the state of two-sensor systems. In one case, each sensor receives or does not receive a signal obstructed by noise. In another, each sensor receives data driven by L\'evy processes with large or small jumps. In either case, these give rise to four possibilities. Infinitesimal generators are presented and analyzed. Bounds for infinitesimal generators in terms of \emph{super-solutions} and \emph{sub-solutions} are computed. An application of this procedure for stochastic model is also presented in relation to the financial market. 
\end{abstract}


\textsc{JEL classifications code:}  C12, C13, C62.

\section{Introduction}
\label{ch:newintro}

One of the most classical problems arising in statistical sequential analysis is the \emph{sequential hypothesis testing} (see \cite{Weld}). As described in \cite{Weld}, a sequential test of a hypothesis means any statistical test that gives a specific rule, at any stage of the experiment for making one of the three decisions: (1) to accept the null hypothesis $H_0$, (2) to reject $H_0$, (3) to continue the experiment by making additional observation. Consequently, the test is carried out \emph{sequentially}. An objective for the analysis of such test is to minimize the number of observations required to make a decision subject to a given tolerance level described as Type I and Type II errors. There are a couple of primary approaches to this problem, e.g.,  the \emph{Bayesian} and the \emph{min-max}. For the first approach, each hypothesis is assigned with an a priori probability. For the second approach, no such assumption is made and the optimal solution is known to be given by the sequential probability ratio test (see \cite{Chow}). The sequential probability ratio test is revisited and improved in various works. For a sequential decision problem, it is assumed that the amount of available information is increasing with time. But it is often difficult to handle all the data as represented by a $\sigma$-algebra as the actual amount may be very large. In \cite{Irle} a reduction method is proposed that takes into account the underlying statistical structure.

An approach of improving the sequential hypothesis testing in the Bayesian case is presented in \cite{Golubev}. The sequential testing of more than two hypotheses has many important applications. In \cite{Baum} a sequential test (termed as MSPRT), which is a generalization of the sequential probability ratio test, is studied. It is shown that, under Bayesian assumptions, the MSPRT approximates optimal tests that are more intricate when error probabilities are small and expected stopping times are large.  In \cite{Dayanik}, a sequential hypothesis test is conducted when there are finitely many simple hypotheses about the unknown arrival rate and mark distribution of a compound Poisson process, where exactly one is correct. This problem is formulated in a Bayesian framework when the objective is to determine the correct hypothesis with minimal error probability. A solution of this problem is presented in that paper. In the paper \cite{Brodsky}, an improved min-max approach to both sequential testing of many composite hypotheses and multi-decision change-point detection for composite alternatives is proposed. New performance measures for methods of hypothesis testing and change-point detection are introduced, and theoretical lower bounds for these performance measures are proved that do not depend on methods of sequential testing and detection. Minimax tests are proposed for which these lower bounds are attained asympototically as decision thresholds tend to infinity. 

In the paper \cite{O} the problem of testing four hypotheses on two streams of observations is examined. Each of the hypotheses is represented by a physical state of presence or absence of a signal obstructed by noise. The objective is to minimize sampling time subject to error probabilities for distinguishing sequentially a standard versus a drifted two-dimensional Brownian motion. This work is based on some key results from \cite{Shiryayev}. A sequential decision rule consisting of a stopping rule that declares the optimal time to stop sampling and a decision variable that declares a decision of the state of our system is formulated in \cite{O}. The rule proposed is the maximum of two sequential probability ratio tests, each with distinct thresholds, and the decision variable is determined from the exit location of the two sequential probability ratio statistics. Thresholds of the proposed rule are computed in terms of the error probabilities. The paper \cite{O} shows a unique way in the construction of a test using a purely two-dimensional structure. This allows detection of a signal in each coordinate rather than merely the detection of a signal somewhere in the system, subject to error probabilities for every possible case. This method can be implemented in a \emph{decentralized} setup and still enjoy the same asymptotic optimality properties as in the other works in the existing literature. This is conducted in three sequential steps: (1) each of the individual sequential probability ratio tests is devised by each of the sensors separately, (2) two sensors communicate a binary bit of information to a central fusion center consisting of the alarm by the sequential probability ratio test and its exit side, and (3) the central fusion center makes a decision after receiving a communication from both sensors. Note that, in this case, the central fusion center does not need to have access to the full two-dimensional stream of sequential observations. In effect, this makes the system faster as  an optimal outcome can be derived with a limited level of communication.

The organization of the paper is as follows. In Section \ref{ch:generalization}, we generalize the result in \cite{O} to a certain class of underlying L\'evy processes, and then create a new decision rule for the size of the jumps in the underlying processes. In Section \ref{ch:bigsmalljumps}, we present the infinitesimal generator for the hypothesis testing of \emph{big} versus \emph{small} fluctuations. This is conducted by considering two L\'evy processes, where one has more \emph{jump intensity} than the other. The viscosity solution and its bounds are analyzed in that section. Finally, a brief conclusion is provided in Section \ref{conclusion}. This section directs to some future research directions based on the present work.

\section{Drift Test Generalization}
\label{ch:generalization}

For various financial time series data, \emph{jumps} play an important role. Jumps in a stochastic model are typically captured by a L\'evy process. Consequently, in this section, we generalize the analysis presented in \cite{O}. We present the analysis for the case when the signals are driven not only by a Brownian motion, but by a generalized L\'evy process. Consequently, the analysis presented is applicable for sequential decision making on the state of a two-sensor system with uncorrelated noise. Each sensor receives or does not receive a signal obstructed by said noise. This gives rise to four possibilities, viz. $\langle$noise, noise$\rangle$ (denoted by 00),  $\langle$signal, noise$\rangle$ (denoted by 10),  $\langle$noise, signal$\rangle$ (denoted by 01), and  $\langle$signal, signal$\rangle$ (denoted by 11). 

The paper \cite{O} seeks to devise a two-sensor hypothesis test based upon a two-dimensional Wiener process $z=(z_1,z_2)$ where 
$$dz_t^{(k)}=\sigma_k dW_t^{(k)}+\mu_k dt,\quad k=1,2,$$
where $W_t^{(k)}$ are Brownian motions with correlation $\rho$. The hypotheses 
\begin{align}\label{genhyp}
\nonumber & \! \!\! \!\! \!\! \! \!\! \!\! \!\! H_{00}: \mu_1=0, \quad \mu_2=0, & H_{10}: \mu_1=m_1\neq 0, \quad \mu_2=0,  \\
 &\! \!\! \!\! \!\! \!\! \!\! \!\! \! H_{01}: \mu_1=0,\quad \mu_2=m_2\neq 0, & H_{11}: \mu_1=m_1\neq 0, \quad \mu_2=m_2 \neq 0.    \! \! \! \! \! \! \! \! \! \!  \! \! \! \! \! \! 
\end{align}
are tested, with decision rule based on the location of the first exit time of the log-likelihood process from a rectangle $R:=(l_1,r_1)\times (l_2,r_2)$, with $l_1<0<l_2$ and $r_1<0<r_2$, when $\rho=0$. Using the infinitesimal generator of the log-likelihood process, the problem is analyzed for the probability of a correct decision in each of the worlds $\{00,01,10,11\}$ as a function of the position of the process in the rectangle. Next, a system of equations in $l_1,l_2,r_1,$ and $r_2$ is constructed by applying rules based on the probabilities of a Type I error. By using the symmetry of the solutions of the equation based on the infinitesimal generator, the bounds of the rectangle are obtained, completing the task of developing a well-defined decision rule. Finally, the optimality of that rule based on minimizing the time required to observe the underlying Wiener process is proved. The following theorem summarizes the main results of \cite{O}:

\begin{theorem}
\label{Oth}
The decision rule \eqref{optdec} when applied to hypotheses \eqref{genhyp} has optimality of order $3$, and choosing three of the four values for Type I errors $\alpha_{ij,0}$ induces the value of the fourth by
$$(1-\alpha_{00,0})(1-\alpha_{11,0})=(1-\alpha_{01,0})(1-\alpha_{10,0}).$$ 
Then, setting the value of $l_1$ will determine the other three, fully defining the decision rule through the system
$$\ln \left( \frac{\alpha_{10,0}-\alpha_{00,0}}{1-\alpha_{00,0}}\right) < l_1 <\ln\left(\frac{\alpha_{10,0}}{1-\alpha_{00,0}}\right),$$
$$r_1=-\ln\left( 1-\frac{1-e^{l_1}}{(1-\alpha_{10,0})/(1-\alpha_{00,0})}\right),$$
$$e^{r_1}=\frac{(1-\alpha_{10,0})/(1-\alpha_{00,0})}{\frac{e^{r_2}}{(1-\alpha_{01,0})+(1-\alpha_{00,0}e^{r_2})}-1},$$
$$r_2=-\ln\left( 1-\frac{1-e^{l_2}}{(1-\alpha_{01,0})/(1-\alpha_{00,0})}\right).$$
\end{theorem}

While Theorem \ref{Oth} provides an optimal decision rule, often we experience processes with jump terms as well. In this paper, we seek to generalize the result in \cite{O} to a certain class of underlying L\'evy processes, and then create a new decision rule for the size of the jumps in the underlying processes. We state a rule to test the hypotheses.

Consider $z=(z_1,z_2)$ a two-dimensional L\'evy process defined by L\'evy triplet $(\mu, \Sigma, \nu^*)$, where $\mu=[\mu_1,\mu_2]$ is the two-dimensional drift, $\Sigma=\left[ \begin{array}{cc} \sigma_1^2 & \rho \\ \rho & \sigma_2^2 \end{array} \right]$ is a symmetric non-negative definite matrix representing the diffusion, and $\nu^*$ is a two-dimensional L\'evy measure defined by a product of two identical one-dimensional L\'evy measures $\nu$. Under this setting, we wish to test the hypotheses \eqref{genhyp}. Note that these hypotheses strictly address the drift terms of the L\'evy process. The primary difference is that we include a L\'evy measure, despite its not changing based on the hypotheses. 

The L\'evy process generates a filtration, which will be denoted $\mathcal{F}_t$, along with marginal filtrations $\mathcal{F}^{(1)}_t$ and $\mathcal{F}^{(2)}_t$. Further, the hypotheses and diffusion correlation $\rho$ induce probability measures $P_{ij,\rho}$ and marginal probability measures $P_i^{(k)}$. We seek to create optimal decision rules $(\tau, \delta_\tau)$, where $\tau$ is a stopping rule with respect to $\mathcal{F}_t$ and $\delta_\tau$ is a random variable taking values in the index set $\{00,01,10,11\}$. Optimality will be based on minimizing the observation time required for given error probabilities $\alpha_{ij,\rho}:=P_{ij,\rho}(\delta_\tau\neq ij)$.

Let 
\begin{equation}\label{logratio}
u^{(i,k)}_t = \log \frac{dP_{i}^{(k)}}{dP_{1-i}^{(k)}}.
\end{equation} 
We define a rectangle $R:=[l_1,r_1] \times [l_2,r_2] \subset \mathbb{R}^2$ and denote
\begin{align} \label{onedec}
\nonumber \tau_k&=\inf \{ t\geq 0: u^{(i,k)}_t \notin [l_k,r_k] \},\\
\nonumber \delta_{\tau_k}^{(i,k)} & = 1-i, \hbox{ if } u^{(i,k)}_{\tau_k}\leq l_k, \\
\delta_{\tau_k}^{(i,k)} & = i, \hbox{ if } u^{(i,k)}_{\tau_k}\geq r_k.
\end{align}
The decision rule for the two-dimensional test is defined as 
\begin{align}\label{optdec}
\nonumber \tau&=\tau_1 \vee \tau_2 \\
\delta_{\tau}^{(i,j)} & =\delta_{\tau_1}^{(i,1)} \delta_{\tau_2}^{(j,2)}.
\end{align}

The following theorem concerns the infinitesimal generator of the processes $u^{(i,k)}_t$:
\begin{theorem}
\label{genlogratio}
Assuming the Brownian motions of the two one-dimensional processes $z_1$ and $z_2$ are uncorrelated and with $u^{(i,k)}_t$ defined as in \eqref{logratio}, we have two-dimensional infinitesimal generators, for the processes $u^{i,j}_t=(u_t^{(i,1)},u_t^{(j,2)})$,
$$\mathcal{L}_{ij,0}=\frac{m_1^2}{2\sigma_1^2}\left(\partial_{x_1x_1} +(-1)^{i+1}\partial_{x_1}\right)+\frac{m_2^2}{2\sigma_2^2}\left(\partial_{x_2x_2} +(-1)^{j+1}\partial_{x_2}\right).$$
\end{theorem}
\begin{proof}
Since $z_k$ is a L\'evy process with characteristics $(0, \sigma_k^2, \nu)$ under $P_0^{(k)}$ and characteristics $(m_k, \sigma_k^2, \nu)$ under $P_1^{(k)}$, we can apply a generalized Girasanov's Theorem (see \cite{Jens} Theorem 1.20). Using $\beta=(-1)^{i+1}m_k/\sigma_k$, we obtain
$$\frac{dP^{(k)}_i}{dP^{(k)}_{1-i}}=\mathcal{E}\left((-1)^{i+1}\frac{m_k}{\sigma_k} W.\right)_t,$$
where $W$ is a standard Brownian motion. Further, by \cite{Tank} (Proposition 8), we obtain characteristics $((-1)^i\frac{m_k^2}{2\sigma_k^2},\frac{m_k^2}{\sigma_k^2},0)$ for  $u^{(i,k)}_t$. Then the process $\left(  u^{(i,1)}_t, u^{(j,2)}_t\right)$ is well-known to have generator 
$$\mathcal{L}_{ij,0}=(-1)^{i+1}\frac{m_1^2}{2\sigma_1^2}\partial_{x_1} +(-1)^{j+1}\frac{m_2^2}{2\sigma_2^2}\partial_{x_2}+ \frac{m_1^2}{2\sigma_1^2}\partial_{x_1x_1}+ \frac{m_2^2}{2\sigma_2^2}\partial_{x_2x_2},$$
as claimed.
\end{proof}
These infinitesimal generators are then used to determine the bounds of the rectangle by applying them to the likelihood functions $\psi_{ij}:R\to [0,1]$  which represents the probability of being in world $ij$ at any position inside the rectangle. When in the correct world, the likelihood function does not change with respect to that world's generator: i.e., $\mathcal{L}_{ij,0}\psi_{ij}=0$.

Next we present a theorem of optimality. The proof of this theorem follows directly from the method in \cite{O}.

\begin{theorem}
The decision rule for \eqref{genhyp} defined in \eqref{optdec} has asymptotic optimality of order-3; that is, 
\begin{eqnarray} E_{ij,0}(\tau_1 \vee \tau_2) - E_{ij}(\tau_1)=o(1),\label{o3}
\end{eqnarray} as the error probabilities $\alpha_{ij,0}\to 0$.
\end{theorem}

\begin{proof}
The proof is similar to the proof of a related result in \cite{O}. The key tool in this proof is the exponential killing trick, i.e., for any nonnegative random variable $Y$ with no point mass at zero,
$$E(e^{-\lambda Y})=E(F_Y(X_\lambda)),$$
where $F_Y$ is the cumulative distribution function for $Y$ and $X_\lambda$ is an exponential random variable with parameter $\lambda$. We use this on $\tau_1$ and $\tau_2$ from \eqref{optdec}:
\begin{eqnarray*}
E_{ij,0}(e^{-\lambda \tau_1}) &=&E_{ij,0}(F_{\tau_1}(X_\lambda)), \\
E_{ij,0}\left(e^{-\lambda (\tau_1\vee \tau_2)}\right)&=&E_{ij,0}(F_{\tau_1}(X_\lambda)F_{\tau_2}(X_\lambda)),
\end{eqnarray*}
as $\tau_1$ and $\tau_2$ are independent when $\rho=0$. Applying Laplace transforms, simplifying, and applying exponential killing in \eqref{o3}, we find
\begin{eqnarray}
 E_{ij,0}(\tau_1 \vee \tau_2) - E_{ij}(\tau_1) =\lim_{\lambda\to 0}\frac{E_{ij,0}(F_{\tau_1}(X_\lambda)(1-F_{\tau_2}(X_\lambda)))}{\lambda}.
\end{eqnarray}
Because our log-likelihood processes are Brownian motions with drift and so their exit times are finite, this difference on the left is finite. Define
$$0<A_k:=e^{l_k}<1<B_k:=e^{r_k}<\infty,$$ and
$$C_1:=\frac{1-\alpha_{10,0}}{1-\alpha_{00,0}}\to 1 \hbox{ as $\alpha_{ij,0}\to 0$},$$
$$C_2:=\frac{1-\alpha_{11,0}}{1-\alpha_{10,0}}\to 1 \hbox{ as $\alpha_{ij,0}\to 0$}.$$
Due to this and \ref{Oth}, $l_k=-r_k$ as $\alpha_{ij,0}\to 0.$

Further, also by \ref{Oth}, $$\frac{C_1+B_1}{B_1} \cdot \frac{C_2+B_2}{B_2}=\frac{1}{1-\alpha_{00,0}}\to 1 \hbox{ as $\alpha_{ij,0}\to 0$.}$$
Because $B_k>1$, we see that $r_k\to \infty$ as $\alpha_{ij,0}\to 0$, and so, $l_k\to -\infty$. Hence, $F_{\tau_k}\to 0$ as well. Finally, applying the dominated convergence theorem,
$$\lim_{\alpha_{ij,0}\to 0}\lim_{\lambda\to 0}\frac{E_{ij,0}(F_{\tau_1}(X_\lambda)(1-F_{\tau_2}(X_\lambda)))}{\lambda}=\lim_{\lambda\to 0} \lim_{\alpha_{ij,0}\to 0}\frac{E_{ij,0}(F_{\tau_1}(X_\lambda)(1-F_{\tau_2}(X_\lambda)))}{\lambda}=0,$$
which yields the desired asymptotic optimality limit.
\end{proof}

\section{Hypothesis Tests on the L\'evy Measure}
\label{ch:bigsmalljumps}

In this section we expand the idea presented in the last section. Before presenting the analysis, we briefly introduce a possible application of this work. A commonly used stochastic model for the derivative market analysis is the Barndorff-Nielsen and Shephard (BN-S) model (see see \cite{BN1, BN-S1, BN-S2, BJS, Hab1, Hab2, Issaka, SWW}). The BN-S model is also implemented in the commodity market (see \cite{SWW, recent}). Though this model is very efficient and simple to use, it suffers from the absence of a long range dependence and many other issues.  Mathematically, for the BN-S model, the stock or commodity price $S= (S_t)_{t \geq 0}$ on some filtered probability space $(\Omega, \mathcal{F}, (\mathcal{F}_t)_{0 \leq t \leq T}, \mathbb{P})$ is modeled by

\begin{equation}
\label{1}
S_t= S_0 \exp (X_t),
\end{equation}
\begin{equation}
\label{2}
dX_t = (\mu + \beta \sigma_t ^2 )\,dt + \sigma_t\, dW_t + \rho \,dZ_{\lambda t}, 
\end{equation}
\begin{equation}
\label{3}
d\sigma_t ^2 = -\lambda \sigma_t^2 \,dt + dZ_{\lambda t}, \quad \sigma_0^2 >0,
\end{equation}
where the parameters $\mu, \beta, \rho, \lambda \in \mathbb{R}$ with $\lambda >0$ and $\rho \leq 0$ and $r$ is the risk-free interest rate where a stock or commodity is traded up to a fixed horizon date $T$.  In this model, $W_t$ is a Brownian motion, and the process $Z_{t}$ is a subordinator. Also $W_t$ and $Z_t$ are assumed to be independent, and $(\mathcal{F}_t)$ is assumed to be the usual augmentation of the filtration generated by the pair $(W_t, Z_t)$. 

In a recent work \cite{recent}, it is shown that for various derivative and commodity price dynamics, the jump is \emph{not} completely stochastic. On the contrary, there is a \emph{deterministic} element in crude oil price that can be implemented in the existing models for an extended period of time. It may be shown that the dynamics of $X_t$ in \eqref{2} can be more accurately written when we use a convex combination of two independent subordinators, $Z$ and $Z^{(b)}$ as:

\begin{equation}
\label{2new}
dX_t = (\mu + \beta \sigma_t ^2 )\,dt + \sigma_t\, dW_t + \rho\left( \theta \,dZ_{\lambda t}+ (1-\theta) dZ^{(b)}_{\lambda t}\right), 
\end{equation}
where $\theta \in [0,1]$ is a \emph{deterministic} parameter. The process $Z^{(b)}$ in \eqref{2new} is a subordinator that has greater intensity than the subordinator $Z$. In this case \eqref{3} will be given by
\begin{equation}
\label{4new}
d\sigma_t ^2 = -\lambda \sigma_t^2 \,dt + \theta' dZ_{\lambda t} + (1-\theta')dZ^{(b)}_{\lambda t} , \quad \sigma_0^2 >0,
\end{equation}
where, as before, $\theta' \in [0,1]$ is \emph{deterministic}.

We observe that even for commonly implemented stochastic models, it is important to detect when a ``smaller" fluctuation ($Z$) turns into a ``larger" fluctuation ($Z^{(b)}$). Consequently, it is important to determine a sequential testing for the analysis of the jump size distribution. The advantages of the dynamics given by the refined BN-S model, given by \eqref{1}, \eqref{2new}, and \eqref{4new}, over existing models are significant. This minor change in the model incorporates \emph{long range dependence} without actually changing the model.

With this in mind, we consider $z=(z_1,z_2)$ a two-dimensional L\'evy process defined by L\'evy triplet $(\mu, \Sigma, \nu^*)$, where $\mu=[\mu_1,\mu_2]$ is the two-dimensional drift, $\Sigma=\left[ \begin{array}{cc} \sigma_1^2 & \rho \\ \rho & \sigma_2^2 \end{array} \right]$ is a symmetric non-negative definite matrix representing the diffusion, and $\nu^*$ is a two-dimensional L\'evy measure defined by a product of two one-dimensional L\'evy measures $\nu_1$ and $\nu_2$ with densities $\nu_k(dx)=(1+\alpha_kx)\nu(dx)$ for some L\'evy measure $\nu$ defined on $\mathbb{R}^+$.

We wish to test the hypotheses
\begin{align}\label{jumphyp}
\nonumber & \! \!\! \!\! \!\! \! \!\! \!\! \!\! H_{00}: \alpha_1=0, \quad \alpha_2=0, & H_{10}: \alpha_1=a_1> 0, \quad a_2=0,  \\
 &\! \!\! \!\! \!\! \!\! \!\! \!\! \! H_{01}: a_1=0,\quad \alpha_1=a_2> 0, &   H_{11}: \alpha_1=a_1> 0, \quad \alpha_1=a_2> 0. \! \!  \!    \! \!  \! \!\! \!  \! \!\! \!  \! \!
\end{align}
These now address the size of the jumps in the L\'evy process.

Similar to last section, the L\'evy process generates a filtration, which will be denoted $\mathcal{F}_t$, along with marginal filtrations $\mathcal{F}^{(1)}_t$ and $\mathcal{F}^{(2)}_t$. Further, the hypotheses and diffusion correlation $\rho$ induce probability measures $P_{ij,\rho}$ and marginal probability measures $P_i^{(k)}$. We seek to create optimal decision rules $(\tau, \delta_\tau)$, where $\tau$ is a stopping rule with respect to $\mathcal{F}_t$ and $\delta_\tau$ is a random variable taking values in the index set $\{00,01,10,11\}$. 

Let $u^{(i,k)}_t$ be defined as in \eqref{logratio}
and still consider a rectangle $[l_1,r_1] \times [l_2,r_2] \subset \mathbb{R}^2$. The decision rules for the non-correlated one-dimensional cases are as in \eqref{onedec}, with the combined decision rule in \eqref{optdec}. 

It is known that if $(X_t)_{t \geq 0}$ is a L\'evy process then there exists a unique c\'adl\'ag process $(Z_t)_{t \geq 0}$ such that 
$$dZ_t= Z_{t-}\, dX_t, \quad Z_0=1.$$
$Z$ is called the stochastic exponential or Dol\'eans-Dade exponential of $X$ and is denoted by $Z= \mathcal{E}(X)$. We can now derive the infinitesimal generators:

\begin{theorem}\label{jumpgen}
With the process $u^{(i,k)}_t$ defined as in \eqref{logratio}, we have two-dimensional infinitesimal generators, for the process $u^{i,j}_t=(u_t^{(i,1)},u_t^{(j,2)})$, defined by
$$\mathcal{L}_{ij,\rho}\xi(x)= (-1)^{i+1}\gamma_1 \xi_{x_1}(x) + (-1)^{j+1}\gamma_2 \xi_{x_2}(x) +\frac{1}{2} \beta_1^2 \xi_{x_1x_1}(x) +\frac{1}{2} \beta_2^2 \xi_{x_2x_2}(x)+\rho \beta_1 \beta_2 \xi_{x_1x_2}(x)$$
$$\quad \quad \quad+(-1)^{i+j}\int_{\mathbb{R}^2_+} \left(\xi(x+y)-\xi(x)-\frac{y_1 \xi_{x_1}(x)+y_2 \xi_{x_2}(x)}{1+\|y\|}\right)K_1(dy_1)K_2(dy_2),$$
for any suitable $\xi$, where 
\begin{align}
\nonumber x&=(x_1,x_2), \quad y=(y_1,y_2)\\
\beta_k&= -a_k\int_{x_k>0} (1\wedge x_k) \sigma_k^{-1}  x_k\nu(dx_k)\label{beta}\\
m_k&=a_k\int_{x_k>1} x_k\nu(dx_k)\label{m}\\
\gamma_k&=m_k-\frac{\beta_k^2}{2}+\int_0^1 (\log(1+x_k)^2-x_k)a_k\nu(dx_k).\label{gamma}\\
K_k&=a_k\log(1+x_k)\nu_k\label{K}. 
\end{align}
\end{theorem}
\begin{proof}
Since $z_k$ is a L\'evy process with characteristics $(\mu_k, \sigma_k^2, \nu_k)$ under $P_0^{(k)}$ and characteristics $(\mu_k, \sigma_k^2, (1+a_kx_k)\nu_k)$ under $P_1^{(k)}$, we apply the generalized Girasanov's Theorem. Using $\beta_k$ as in \eqref{beta}, we obtain
$$\frac{dP^{(k)}_i}{dP^{(k)}_{1-i}}=\mathcal{E}\left((-1)^{i+1}N^{(k)}.\right)_t,$$
where
$$N^{(k)}_t=\beta_k W_t +\int_0^t\int_{x_k>0} a_kx_k (J_k-\nu_k)(ds,dx_k),$$
$a_kJ_k$ is the jump measure for $N$, $W$ is a standard Brownian motion, and $\mathcal{E}$ is the Dol\'eans-Dade exponential. This gives that $N_t^{(k)}$ is a L\'evy process with characteristics $$(  (-1)^{i}m_k, \beta_k^2,(-1)^{i+1} a_k\nu_k).$$ Then, by \cite{Tank} (Proposition 8), we obtain characteristics $$((-1)^{i}\gamma_k,\beta_k^2,(-1)^{i+1}K_k)$$ for $u^{(i,k)}_t$. Finally, by \cite{Tank, AS}, the process $\left( u^{(i,1)}_t, u^{(j,2)}_t \right)$ has the stated generator.
\end{proof}
Assign $\xi_{ij,\rho}$ to be the probability of a correct decision in world $ij$. Then we have the partial integro-differential equation $\mathcal{L}_{ij,\rho}\xi_{ij,\rho}=0$ with boundary conditions
\begin{align}
\xi_{00,\rho}(l_1,l_2)=1, \quad \quad \quad & \xi_{01,\rho}(l_1,r_2)=1, \nonumber \\
\xi_{00,\rho}(r_1,y)=0,  \quad \quad \quad & \xi_{01,\rho}(x,l_2)=0, \nonumber \\
\xi_{00,\rho}(x,r_2)=0,  \quad \quad \quad & \xi_{01,\rho}(r_1,y)=0, \nonumber  \\
\xi_{10,\rho}(r_1,l_2)=1, \quad \quad \quad & \xi_{11,\rho}(r_1,r_2)=1, \nonumber  \\
\xi_{10,\rho}(l_1,y)=0,  \quad \quad \quad& \xi_{11,\rho}(l_1,y)=0,  \nonumber \\
\xi_{10,\rho}(x,r_2)=0, \quad \quad \quad & \xi_{11,\rho}(x,l_2)=0, \label{bcs}
\end{align}
for $x \in [l_1,r_1]$ and $y \in [l_2,r_2]$. Further, we have $ \xi_{ij,\rho}>0$ inside $R=(l_1,r_1) \times (l_2,r_2)$.

Before proceeding to prove the existence of a solution to such a boundary value problem, we state the following definitions and a theorem from 
\cite{EU} that will be used.
\begin{definition}
An upper semicontinuous function $l: \R^2\to \R$ is a \emph{subsolution} of $$F(0,\xi,D\xi,D\xi^2,\mathcal{I}[\xi](x))=0$$ subject to boundary conditions \eqref{bcs} if for any test function $\phi\in C^2(\R^2)$, at each maximum point $x_0 \in \bar{R}$ of $l-\phi$ in $B_\delta(x_0)$, we have 
$$E(l,\phi,x_0):=F(x_0,l(x_0),D\phi(x_0),D^2\phi(x_0), I_\delta^1[\phi](x_0)+I_\delta^2[l](x_0))\leq 0 \hbox{   if $x_0\in R$}$$
or
$$\min(E(l,\phi,x_0);u(x_0)-g(x_0))\leq 0 \hbox{  if $x_0\in \partial R$},$$
where
\begin{align*}
I_\delta^1[\phi](x_0)=\int_{|z|<\delta} \left( \phi(x_0+z)-\phi(x_0)-(D\phi(x_0) \cdot z)\textbf{1} _B (z)\right) d\mu_{x_0}(z), \\
I_\delta^2[u](x_0)=\int_{|z|\geq\delta} \left( u(x_0+z)-u(x_0)-(D\phi(x_0) \cdot z)\textbf{1}_B (z)\right) d\mu_{x_0}(z) .
\end{align*}

Similarly, a lower semicontinuous function $u: \R^2 \to \R$ is a \emph{supersolution} of the same boundary value problem if for any test function $\phi\in C^2(\R^2)$, at each minimum point $x_0 \in \bar{R}$ of $u-\phi$ in $B_\delta(x_0)$, we have 
$$E(u,\phi,x_0)\geq 0\hbox{   if $x_0\in R$}$$
or 
$$\max(E(l,\phi,x_0);u(x_0)-g(x_0))\leq 0 \hbox{  if $x_0\in \partial R$}.$$
Finally, a \emph{viscosity solution} is a function whose upper and lower semicontinuous envelopes are respectively a \emph{sub-solution} and a \emph{super-solution}.
\end{definition}

\begin{theorem}\label{EUtheorem}
If $F:\R^2 \times \R \times \R^2 \times \mathbb{S}^2 \times \R \to \R$, where $\mathbb{S}^n$ is the space of $n\times n$ symmetric matrices, and
\begin{enumerate}
\item[(A1)] $F(x,u,p,X,i_1)\leq F(x,u,p,Y,i_2)$ if $X\geq Y$ and $i_1\geq i_2$,
\item[(A2)] there exists $\gamma>0$ such that for any $x\in \R^2,$ $u,v\in \R$, $p\in \R^2$, $X\in \mathbb{S}^2$, and $i\in \R$,
$$F(x,u,p,X,i)-F(x,v,p,X,i)\geq \gamma (u-v) \hbox{   if $u \geq v$},$$
for some $\epsilon>0$ and $r(\beta)\to 0$ as $\beta \to 0$, we have
$$F(y,v,\epsilon^{-1}(x-y),Y,i)-F(x,v,\epsilon^{-1}(x-y),X,i)\leq \omega_R (\epsilon^{-1}|x-y|^2+|x-y|+r(\beta)),$$
\item[(A3)] $F$ is uniformly continuous with respect to all arguments,\\
\item[(A4)] $\sup_{x\in \R} |F(x,0,0,0,0)|<\infty$,\\
\item[(A5)] $K=K_1 \times K_2$ is a L\'evy-It\^o measure,
\item[(A6)] the inequalities in \eqref{limsupinf} are strict,
\item[(A7)] for any $R>0$, there exists a modulus of continuity $\omega_R$ such that, for any $x,y\in\R^2$, $|v|\leq R$, $i\in \R$, and for any $X,Y\in \S^2$ satisfying
$$\left[\begin{array}{cc} X & 0 \\ 0 & Y\end{array}\right] \leq \frac{1}{\epsilon}\left[\begin{array}{cc} I & -I \\ -I & I \end{array}\right]+r(\beta)\left[\begin{array}{cc} I & 0 \\ 0 & I \end{array}\right],$$
\end{enumerate}
then there is a unique solution to $F(0,\xi,D \xi, D^2 \xi,\mathcal{I}[\xi](x))=0$ between any pair of super-solution and sub-solutions, defined below, where 
$$\mathcal{I}[\xi](x):=\int_{\mathbb{R}^2_+} \left(\xi(x+y)-\frac{y_1 \xi_{x_1}(x)+y_2 \xi_{x_2}(x)}{1+\|y\|}\right)K_1(dy_1)K_2(dy_2).$$
\end{theorem}

\begin{lemma}
In particular, our function $$F(x,u,p,X,i):=Mu+\langle \gamma_1 ,\gamma_2\rangle \cdot p -\Tr \left( \left[ \begin{array}{cc} \beta_1/2 & 0 \\ 0 & \beta_2/2 \end{array}\right]X \right)- i$$ satisfies (A1)-(A4)  and our measure $K$ satisfies (A5) in \eqref{EUtheorem}, where 
$$M=\int_{\mathbb{R}^2_+}K_1(dy_1)K_2(dy_2).$$
\end{lemma}
\begin{proof}
First, consider (A1): 
$$F(x,u,p,X,i_1)- F(x,u,p,Y,i_2) =  \Tr \left( \left[ \begin{array}{cc} \beta_1/2 & 0 \\ 0 & \beta_2/2 \end{array}\right](Y-X) \right)+i_2- i_1 \geq0$$
if $i_2\leq i_1$ and $Y\leq X$.

Next, $F(x,u,p,X,i)- F(x,v,p,X,i)=M(u-v)$, so choosing $\gamma=M >0$, we have property (A2).

Property (A3) is satisfied because $F$ is linear in each argument, and (A4) is satisfied because $F$ does not depend on its first argument explicitly. Last, $K$ is a L\'evy-It\^o measure by the assumptions of the underlying L\'evy processes. 

\end{proof}

Note that the $F$ above corresponds to case $ij=00$. The other three cases can be similarly satisfied through manipulation of the signs in $F$. Before proceeding, we present a few formal definitions:
\begin{definition}
We write that a function $f(x)=O(g(x))$  if we have some $M,\epsilon \in \R$ satisfying $|f(x)|\leq Mg(x)$ for all $x>\epsilon$. Similarly, we write that a function $f(x)=o(g(x))$ if we have some $M,\epsilon \in \R$ satisfying $|f(x)|< Mg(x)$ for all $x>\epsilon$.

The norm $\| f \|_\infty$ is defined as the essential supremum of the absolute value of $f$ over $\Omega$. It is the smallest number so that $\{ x : |f(x)|\geq \|f\|_\infty\}$ has measure zero.
\end{definition}

We now state the additional limit assumptions on $F$ from \cite{EU}: 
\begin{align} \liminf_{y\to x, y\in \bar{\Omega}, \eta \downarrow 0, d(y)\eta^{-1} \to 0} \left[ \sup_{0<\delta \in [d(y),r)} \inf_{s\in [-R,R]} F(y,s,p_\eta(y),M_\eta(y),I_{\eta,\delta,r}(y))\right]&<0, \nonumber \\ 
\limsup_{y\to x, y\in \bar{\Omega}, \eta \downarrow 0, d(y)\eta^{-1} \to 0} \left[ \inf_{0<\delta \in [d(y),r)} \sup_{s\in [-R,R]} F(y,s,-p_\eta(y),-M_\eta(y),-I_{\eta,\delta,r}(y))\right]&<0, \label{limsupinf}\end{align}
where
\begin{equation*}
p_\eta(y) =O(\epsilon^{-1})+\frac{k_1+o(1)}{\eta} Dd(y),
\end{equation*}
\begin{equation*}
M_\eta(y)= O(\epsilon^{-1})+\frac{k_1+o(1)}{\eta} D^2d(y)-\frac{k_2+o(1)}{\eta^2} Dd(y) \otimes Dd(y),
\end{equation*}
\begin{align*}
I_{\eta,\delta,r}(y)&=-\nu I_{\delta,r}^{\hbox{ext},1}(y)+2\|u\|_\infty I_{\beta(\nu),r}^{\hbox{int},1}(y)\\
&-\frac{k_1+o(1)}{\eta} \left(I^{\hbox{tr}}(y)+I_{\beta(\eta),r}^{\hbox{int},2}(y)+I_{\delta,r}^{\hbox{ext},2}(y)-\|D^2d\|_\infty I^4_{\delta,\beta(\eta),r}(y)\right)\\
&+O(\epsilon^{-1})\left( 1+ o(1)I_{\beta(\eta),r}^{\hbox{int},3}(y)+o(1)I_{\delta,r}^{\hbox{ext},3}(y) \right),
\end{align*}
with $O(\epsilon^{-1})$ not depending on $k_1$ nor $k_2$, and
\begin{equation*}
\mathcal{A}_{\delta,\beta,r}(x) :=\{z\in B_r: -\delta\leq d(x+z)-d(x)\leq \beta\},
\end{equation*}
\begin{equation*}
\mathcal{A}^{\hbox{ext}}_{\delta,r}(x):=\{z\in B_r:  d(x+z)-d(x)< -\delta\},
\end{equation*}
\begin{equation*}
\mathcal{A}^{\hbox{int}}_{\beta,r}(x):=\{z\in B_r:  d(x+z)-d(x)>\beta\},
\end{equation*}
\begin{equation*}
I^{\hbox{ext},1}_{\delta,r}(x) := \int_{\mathcal{A}^{\hbox{ext}}_{\delta,r}(x)} d\mu_x(z),
\end{equation*}
\begin{equation*}
I^{\hbox{ext},2}_{\delta,r}(x) := \int_{\mathcal{A}^{\hbox{ext}}_{\delta,r}(x)} Dd(x)\cdot zd\mu_x(z),
\end{equation*}
\begin{equation*}
I^{\hbox{ext},3}_{\delta,r}(x) := \int_{\mathcal{A}^{\hbox{ext}}_{\delta,r}(x)}|z| d\mu_x(z),
\end{equation*}
\begin{equation*}
I^{\hbox{int},1}_{\beta,r}(x) := \int_{\mathcal{A}^{\hbox{ext}}_{\beta,r}(x)} d\mu_x(z),
\end{equation*}
\begin{equation*}
I^{\hbox{int},2}_{\beta,r}(x) := \int_{\mathcal{A}^{\hbox{ext}}_{\beta,r}(x)} Dd(x)\cdot zd\mu_x(z),
\end{equation*}
\begin{equation*}
I^{\hbox{int},3}_{\beta,r}(x) := \int_{\mathcal{A}^{\hbox{ext}}_{\beta,r}(x)}|z| d\mu_x(z),
\end{equation*}
\begin{equation*}
I^{4}_{\delta,\beta,r}(x) := \frac{1}{2}\int_{\mathcal{A}_{\delta,\beta,r}(x)}|z|^2 d\mu_x(z),
\end{equation*}
\begin{equation*}
I^{\hbox{tr}}(x) := \int_{r<|z|<1} Dd(x)\cdot zd\mu_x(z).
\end{equation*}

\begin{theorem}\label{solution}
The partial integro-differential equation $\mathcal{L}_{ij,0}\xi_{ij,0}=0$ subject to boundary conditions \eqref{bcs} and  $0< \xi_{ij,\rho}$ has a viscosity solution between sub-solution and super-solution
\begin{align*}
L_{ij}(x,y)&= E_{ij}\frac{\sinh \left(A_{(1,2-i)} \sqrt{B_{ij(1)}^2+L}\right) \sinh \left(A_{(2,2-j)} \sqrt{B_{ij(2)}^2+L}\right)}{\sinh\left(\frac{r_1-l_1}{\beta_1} \sqrt{B_{ij(1)}^2+L} \right)\sinh\left(\frac{r_2-l_2}{\beta_2} \sqrt{B_{ij(2)}^2+L}\right)},\nonumber \\
U_{ij}(x,y)&= E_{ij}\frac{\sinh \left(A_{(1,2-i)} \sqrt{B_{ij(1)}^2-L}\right) \sinh \left(A_{(2,2-j)} \sqrt{B_{ij(2)}^2-L}\right)}{\sinh\left(\frac{r_1-l_1}{\beta_1} \sqrt{B_{ij(1)}^2-L} \right)\sinh\left(\frac{r_2-l_2}{\beta_2} \sqrt{B_{ij(2)}^2-L}\right)},
\end{align*}
where 
\begin{align*}
A&=\left[ \begin{array}{cc} \frac{x-l_1}{\beta_1} & \frac{r_1-x}{\beta_1} \\ \frac{y-l_2}{\beta_2} & \frac{r_2-y}{\beta_2} \end{array} \right],\\
B_{ij}&=\left[ \begin{array}{cc} \frac{\gamma_1+(-1)^jC}{\beta_1} & \frac{\gamma_2+(-1)^iC}{\beta_2}  \end{array} \right], \nonumber\\
C&=\int_{\mathbb{R}^2_+} \frac{y_i}{1+\|y\|}K_1(dy_1)K_2(dy_2),\nonumber\\
E_{ij}&=\exp\left(A_{(1,i+1)} B_{ij(1)}+A_{(2,j+1)} B_{ij(2)}\right),\nonumber 
\end{align*}
and $L$ is a positive constant, provided (A6) and (A7) are satisfied. 
\end{theorem}
\begin{proof}
We define
\begin{align*}
H(x)&=\int_{\mathbb{R}^2_+} \xi(x+y) K_1(dy_1)K_2(dy_2),\\
\xi(x)&=f(x_1)g(x_2). \nonumber
\end{align*}
Consequently,
\begin{align}
0=\mathcal{L}_{ij,0}\xi(x)=& (-1)^{i+1}\gamma_1 \xi_{x_1}(x) + (-1)^{j+1}\gamma_2 \xi_{x_2}(x) +\frac{1}{2} \beta_1^2 \xi_{x_1x_1}(x) +\frac{1}{2} \beta_2^2 \xi_{x_2x_2}(x)\nonumber\\
&+(-1)^{i+j}\int_{\mathbb{R}^2_+} \left(\xi(x+y)-\xi(x)-\frac{y_1 \xi_{x_1}(x)+y_2 \xi_{x_2}(x)}{1+\|y\|}\right)K_1(dy_1)K_2(dy_2)\nonumber
\end{align}
can be rewritten as
\begin{align}
0=& (-1)^{i+1}\gamma_1 f'(x_1)g(x_2) + (-1)^{j+1}\gamma_2 f(x_1)g'(x_2) +\frac{1}{2} \beta_1^2 f''(x_1)g(x_2) +\frac{1}{2} \beta_2^2 f(x_1)g''(x_2)\nonumber\\
&+(-1)^{i+j}H(x)+(-1)^{i+j+1}Mf(x_1)g(x_2)\nonumber \\
&+(-1)^{i+j+1}Cf'(x_1)g(x_2)+(-1)^{i+j+1}Cf(x_1)g'(x_2).\nonumber
\end{align}

When $ij\in \{00,11\}$, the sign on $H$ is positive; therefore, we have sub-solution equations, through separation of variables,
\begin{align*}
&\frac{1}{2} \beta_1^2 \frac{f''(x_1)}{f(x_2)}+\left((-1)^{i+j+1}C+(-1)^{i+1}\gamma_1 \right)\frac{f'(x_1)}{f(x_1)} =-\lambda_{1,ij}+(-1)^{i+j}M, \\ 
&\frac{1}{2} \beta_2^2 \frac{g''(x_2)}{g(x_2)}+\left((-1)^{i+j+1}C+(-1)^{j+1}\gamma_2\right)\frac{g'(x_2)}{g(x_2)}=\lambda_{1,ij}.\nonumber
\end{align*}
Alternatively, when $ij\in \{01,10\}$, we have these as super-solution equations.

On the other hand, since $\xi >0$ inside $R$, there exists some $K>0$ so that
$$\xi(x+y)-\xi(x)\leq K\xi(x) \iff H(x)-\int \xi(x) K_1(dy_1)K_2(dy_2)\leq KMf(x_1)g(x_2).$$
Using this, in cases $ij\in\{00,11\}$, we have super-solution equations
\begin{align*}
&\frac{1}{2} \beta_1^2 \frac{f''(x_1)}{f(x_2)}+\left((-1)^{i+j+1}C+(-1)^{i+1}\gamma_1 \right)\frac{f'(x_1)}{f(x_1)}  =-\lambda_{2,ij}+(-1)^{i+j+1}KM, \\ 
&\frac{1}{2} \beta_2^2 \frac{g''(x_2)}{g(x_2)}+\left((-1)^{i+j+1}C+(-1)^{j+1}\gamma_2\right)\frac{g'(x_2)}{g(x_2)}=\lambda_{2,ij}.\nonumber
\end{align*}
When $ij\in \{01,10\}$, we have these as sub-solution equations instead.

Now, choosing $L=\max\{KM,M\}$ and $\lambda_{k,ij}$ to yield $\pm L/2$ on the right-hand sides, the boundary-value problem gives the super-solution and sub-solutions claimed. Due to the monotonicity of $\sinh$ and $\exp$, we see that the super-solution and sub-solutions are ordered so that $ U_{ij}\geq L_{ij}\geq0$. Finally, by \cite{EU}, we have existence of a viscosity solution to  $\mathcal{L}_{ij,\rho}\xi_{ij,\rho}=0$ with boundary conditions \eqref{bcs}.
\end{proof}

\begin{remark}
Note that due to the structure of $U_{ij}$ and $L_{i,j}$, as $L/B_{ij}^2 \to 0$, the super-solution and sub-solutions tend toward each other. While this cannot occur precisely, it grants a particularly interesting condition that can reduce the size of the rectangle used in the decision rule.
\end{remark}

In the following figures (Figures 1, 2, and 3), we plot a super-solution and a sub-solution for a special case. Figure 1 depicts the monotonic nature of $U_{00}$ and $L_{00}$, and shows the boundary conditions are met. Note that the sides of the rectangle, and therefore the bounds of the domain of $x$ here, would be chosen in such a way as to have $1-\alpha_{00}$ between the graphs along the line $x=0$ in Figure 2. Finally, Figure 3 fully shows the super-solution and sub-solution of a viscosity solution.

\begin{figure}[!htb]
   \begin{minipage}{0.45\textwidth}
     \centering
     \includegraphics[width=.9\linewidth]{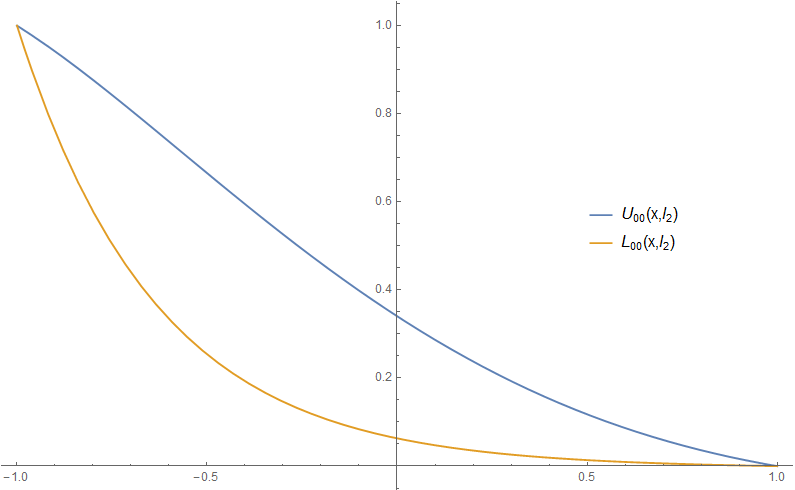}
     \caption{The $y=l_2=-1$ cross-section, with boundary conditions met.}\label{Fig:Data1}
   \end{minipage}\hfill
   \begin{minipage}{0.45\textwidth}
     \centering
     \includegraphics[width=.9\linewidth]{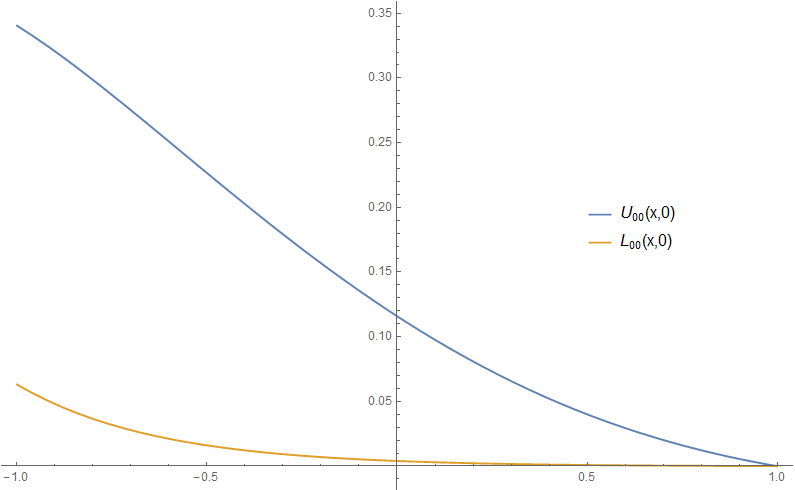}
     \caption{The $y=0$ cross-section, showing the region that determines the actual bounds of $R$.}\label{Fig:Data2}
   \end{minipage}
\end{figure}
\begin{figure}[!h]
     \centering
     \includegraphics[width=.9\linewidth]{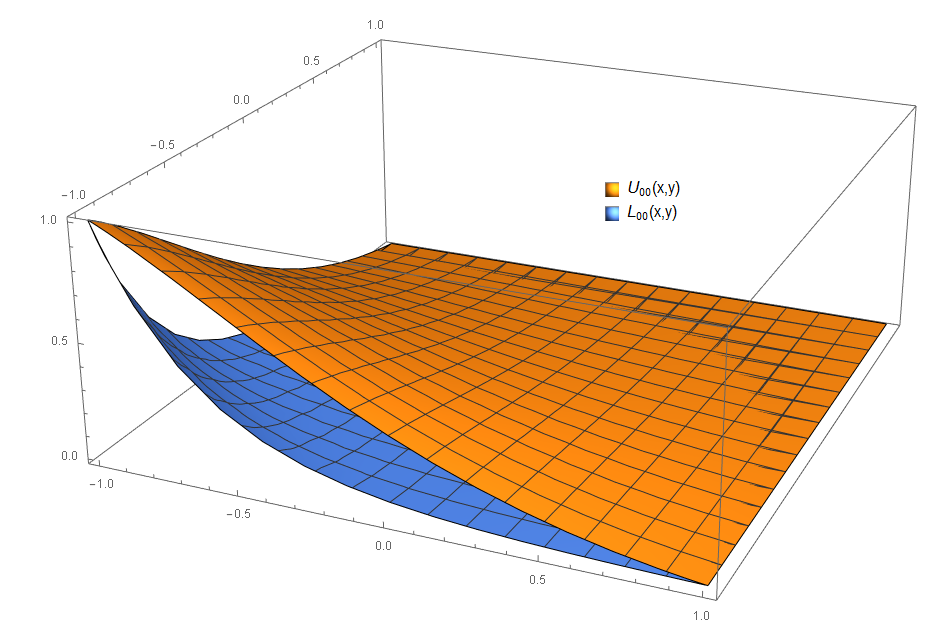}
     \caption{The functions $U_{00}$ and $L_{00}$, which fully envelope the viscosity solution $\xi$.}\label{Fig:Data1}
 \end{figure}

The super-solution and sub-solutions depend on the hypothesis parameters $a_k$. With regards to conducting the hypothesis test, one option is to bound $a_k<\mathcal{A}$. After doing so, 
$U_{ij}$ and $L_{ij}$ create two sets of inequalities  $$\{ U_{ij}(0,0)\geq 1-\alpha_{ij} : ij\in \{00,01,10,11\}\},$$
$$\{ L_{ij}(0,0)\leq 1-\alpha_{ij} : ij\in \{00,01,10,11\}\}.$$
These can potentially be solved for upper and lower estimates for the sides of the rectangle $(l_1,r_1)\times (l_2,r_2)$. While no longer optimal, these estimates can be used to conduct hypothesis tests with Type I error probability at most $\alpha_{ij}$.

\section{One-dimensional Application}

In this section, we implement the analysis presented in the last section for oil from Bakken region. The Bakken is the United States oil producing region that has emerged in recent years and expanded in part due to fracking technology.  Data are collected on oil prices and volumes representative of oil produced in that region. We obtain the average daily production data for each month from the Energy Information Agency.  Also, we obtain the nearby daily Bakken freight-on-board (FOB) prices from Thomson Reuters Eikon. 

At first we consider the Bakken oil data for the period  June 2, 2009 through March 4, 2014 (see Figure 4). 
From the data, we obtain that the statistics are as follows: drift $\mu=0.0238$, volatility $\sigma=11.419$, and the test statistic $a=19.00$.  Choosing $\alpha_0=0.9$ and $l=-0.03$, we can determine $r=\max\{0.3769,0.2569\}=0.3769$. Using the methods from the previous sections, we convert to the log-likelihood ratio process, and then we simulate it in Python. The process exits the interval on the right side 6 out of the 30 times run, indicating that the jumps have statistically low impact on the overall structure of the process.

Next we consider the Bakken oil data for the period May 24, 2011 through May 30, 2019 (see Figure 5). Conducting similar analyses for a shifted set of dates, we see that the statistics are drift $\mu=-0.0278$, volatility $\sigma=23.053$, and test statistic $a=30.43$. Choosing  $\alpha_0=0.9$, we then choose $l=-0.1$, and we find $r=\max\{0.1144,0.1017\}=0.1144$. Simulating the log-likelihood process now reveals the process exits on the right side 12 out of the 30 times run, indicating that while the jumps still have low impact on the overall structure, they have more of an impact than in the previous example, as is expected.

\begin{figure}[h!]
      \includegraphics[width=1\linewidth]{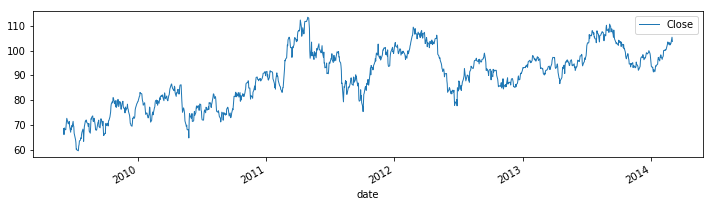}
     \caption{Bakken oil close prices for June 2, 2009-March 4, 2014.}\label{Fig:oil}
 \end{figure}

\begin{figure}[h!]
      \includegraphics[width=1\linewidth]{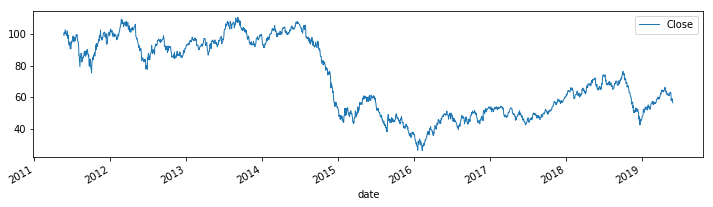}
     \caption{Bakken oil close prices for May 24, 2011-May 30, 2019.}\label{Fig:oil}
 \end{figure}

\section{Conclusion}
\label{conclusion}

In this paper we have studied a sequential decision making problem in connection to the L\'evy processes. Sequential decision making describes a situation where the decision maker makes successive observations of a process before a final decision is made. The procedure to decide when to stop taking observations and when to continue is called the \emph{stopping rule}.  This problem can be implemented for financial derivative or commodity markets. A single stochastic model may not appropriately represent derivative or commodity market dynamics. However, the procedure presented in this paper can be incorporated to determine the fluctuations in the \emph{jump term} of the L\'evy processes. Consequently, the \emph{jump term} can be replaced or modified. Thus with a minor adjustment, the original model becomes more effective.  This modification also enables \emph{long range dependence} in the new model without significantly changing the model. 

The objective in a typical sequential decision making is to find a stopping rule that optimizes the decision in terms of some loss function. For the present paper, the case $\rho=0$ is considered. The situation becomes much more involved when $\rho \neq 0$. This will be considered in a sequel of the present paper. Further, it is worth investigating whether some method exists to determine the exact bounds of the rectangle used in a decision rule based on a super-solution or sub-solution, after one bound is chosen specifically. Even more generally, additional hypothesis tests could be developed. One of such could be a test on the L\'evy measures while keeping a constant diffusion coefficient for each underlying L\'evy process. Another could be a test on the diffusion terms with no drift terms in either process.  Last, it is worth exploring the one-dimensional test more, which could yield a solution useful for finding final boundary conditions for the two-dimensional test, giving uniqueness of the likelihood function. \\

\textbf{Acknowledgment}: The authors would like to thank the anonymous reviewers for their careful reading of the manuscript and for suggesting points to improve the quality of the paper.

\end{document}